\documentclass[11pt]{amsart}

\usepackage{amsfonts,amssymb,amsmath,amsthm,euscript,indentfirst} 
\usepackage{mathtools} 

\numberwithin{equation}{section}
\newtheorem{theorem}{Theorem}[section]
\newtheorem{proposition}[theorem]{Proposition}
\newtheorem{corollary}[theorem]{Corollary}

\newtheorem{lemma}[theorem]{Lemma}

\theoremstyle{definition}

\theoremstyle{remark}
\newtheorem{remark}[theorem]{Remark}

\DeclarePairedDelimiter\ceil{\lceil}{\rceil}
\DeclarePairedDelimiter\floor{\lfloor}{\rfloor}
\DeclareMathOperator{\CP}{\mathbb{C}\mathrm{P}}
\DeclareMathOperator{\ord}{ord}
\DeclareMathOperator{\rk}{rk}
\DeclareMathOperator{\tr}{tr}

\DeclareMathOperator{\Hess}{Hess}

\DeclareMathOperator{\HM1}{\mathit{HM}\!{}_1 (\mathit{n}\!+\!1)}
\newcommand*{\abs}[1]{\left|#1\right|} 
\newcommand*{\brr}[1]{\left(#1\right)} 
\newcommand*{\brc}[1]{\left\{#1\right\}} 
\newcommand{\area}{\mathrm{Area}}

\begin{document}

\title{The first eigenvalue of the Laplacian on orientable surfaces}

\begin{abstract}

The famous Yang-Yau inequality provides an upper bound for the first eigenvalue of the Laplacian on an orientable Riemannian surface solely in terms of its genus $\gamma$ and the area. Its proof relies on the existence of holomorhic maps to $\mathbb{CP}^1$ of low degree. Very recently, Ros was able to use certain holomorphic maps to $\mathbb{CP}^2$ in order to give a quantitative improvement of the Yang-Yau inequality for $\gamma=3$. In the present paper, we generalize Ros' argument to make use of holomorphic maps to $\mathbb{CP}^n$ for any $n>0$. As an application, we obtain a quantitative improvement of the Yang-Yau inequality for all genera except for $\gamma = 4,6,8,10,14$.

\end{abstract}
\author[M. Karpukhin]{Mikhail Karpukhin}
\address{Mathematics 253-37, Caltech, Pasadena, CA 91125, USA
}
\email{mikhailk@caltech.edu}
\author[D. Vinokurov]{Denis Vinokurov}
\address{Faculty of Mathematics, HSE University, Usacheva str. 6, 119048 Moscow, Russia}
\email{vinokurov.den.i@gmail.com}
\maketitle

\section{Introduction}

\subsection{Yang-Yau inequality} Let $(\Sigma,g)$ be a closed orientable Riemannian surface. The Laplace-Beltrami operator, or Laplacian, is defined as $\Delta_g = \delta_g d$, where $\delta_g$ is the formal adjoint of the differential $d$. For closed manifolds the spectrum of $\Delta_g$ consists only of eigenvalues and forms the following sequence
$$
0 = \lambda_0(\Sigma,g)<\lambda_1(\Sigma,g)\leqslant \lambda_2(\Sigma,g)\leqslant\ldots\nearrow\infty,
$$  
where eigenvalues are written with multiplicities.

Consider the normalized eigenvalues
$$
\bar\lambda_k(\Sigma,g) = \lambda_k(\Sigma,g)\area(\Sigma,g).
$$
The problem of geometric optimization of eigenvalues consists in determining the exact values of the following quantities
$$
\Lambda_k(\Sigma) = \sup_g\bar\lambda_k(\Sigma,g).
$$
We refer to~\cite{KNPP,KRP2} and references therein for a detailed survey of recent developments on the problem.

In the present paper we focus on the case $k=1$. The first general upper bound on $\Lambda_1(\Sigma)$ was obtained by Yang and Yau in~\cite{YY} who proved that if $\Sigma$ has genus $\gamma$, then
$$
\Lambda_1(\Sigma)\leqslant 8\pi (\gamma + 1)
$$
However, it was soon remarked in~\cite{ESI} that the same proof yields the following improved bound
\begin{equation}
\label{YY:ineq}
\Lambda_1(\Sigma)\leqslant 8\pi\floor*{\frac{\gamma+3}{2}},
\end{equation}
where $\floor*{x}$ is the floor function or the integer part of $x$. In the following we refer to~\eqref{YY:ineq} as the Yang-Yau inequality. Since the original paper~\cite{YY}, alternative proofs of~\eqref{YY:ineq} have appeared in~\cite{BLY, LY}. 

According to the results of Hersch~\cite{Hersch} and Nayatani-Shoda~\cite{NayataniShoda}, the Yang-Yau inequality is sharp for $\gamma=0$ and $\gamma=2$ respectively. Apart from that, the exact value of $\Lambda_1(\mathbb{T}^2)$ was computed by Nadirashvili in~\cite{Nadirashvili_torus},
$$
\Lambda_1(\mathbb{T}^2) = \frac{8\pi^2}{\sqrt{3}}.
$$
At the same time, it is known that for $\gamma\ne 0,2$ the Yang-Yau inequality is strict, see~\cite{KYY}. The existence of metrics achieving $\Lambda_1(\Sigma)$ has been recently established in~\cite{MS}.

\begin{remark}
The problem of determining $\Lambda_1(\Sigma)$ makes sense for non-orientable $\Sigma$, see~\cite{Karpukhin_nor} for a generalization of~\eqref{YY:ineq} to this setting. We refer to~\cite{LY} and~\cite{EGJ, JNP, CKM} for the exact values of $\Lambda_1$ on the projective plane and Klein bottle respectively. The existence result of~\cite{MS} continues to hold for non-orientable surfaces.
\end{remark}

Finally, in a recent paper~\cite{Ros:main article} Ros obtained a quantitative improvement of~\eqref{YY:ineq} for $\gamma=3$. Namely, he proved that if $\Sigma$ is an orientable surface of genus $3$, then
$$
\Lambda_1(\Sigma)\leqslant 16(4-\sqrt{7})\pi\approx 21.668\pi<24\pi,
$$
where $24\pi$ is the bound given by the Yang-Yau inequality. The results of the present paper are heavily inspired by the work of Ros and are essentially a generalization of~\cite{Ros:main article} to higher genera.

\subsection{Main results} Our main result can be stated as follows.

\begin{theorem}\label{th:the main theorem}
Let $\Sigma_\gamma$ be a compact orientable surface of genus $\gamma$. Then one has
\begin{equation}
\label{main:ineq}
\Lambda_1(\Sigma_\gamma) \leqslant \frac{2\pi}{13-\sqrt{15}} \brr{\gamma + \brr{33-4\sqrt{15}}\ceil*{\frac{5\gamma}{6}} + 4\brr{41-5\sqrt{15}}},
\end{equation}
where $\lceil\cdot\rceil$ is the ceiling function. In particular,
$$
\Lambda_1(\Sigma_\gamma) < 8\pi(0.43\gamma + 2.86).
$$
\end{theorem}

We show in Lemma~\ref{lemma:n<5} that the bound~\eqref{main:ineq} is an improvement over the Yang-Yau inequality~\eqref{YY:ineq} as soon as $\gamma\geqslant 25$. Furthermore, it is easy to compute the asymptotic behaviour of the r.h.s in~\eqref{main:ineq}. We formulate the corresponding result as follows.

\begin{corollary}
Let $\Sigma_\gamma$ be a compact orientable surface of genus $\gamma$. Then one has
\begin{equation}
\label{limsup:ineq}
\limsup_{\gamma\to\infty} \frac{\Lambda_1(\Sigma_\gamma)}{8\pi\gamma} \leqslant \frac{5}{6} -  \frac{89}{6(52-4\sqrt{15})}\approx 0.42703.
\end{equation}
\end{corollary}
At the same time, the Yang-Yau inequality only yields $\frac{1}{2}$ in the r.h.s of~\eqref{limsup:ineq}.
 
Finally, we remark that the bound~\eqref{main:ineq} is a particular member of a family of inequalities proved in  Proposition~\ref{ngamma:prop}. We choose to state Theorem~\ref{th:the main theorem} in its present form due to the fact that~\eqref{main:ineq} is the best bound in the family for $\gamma\geqslant 102$ and, thus, yields the best constant in~\eqref{limsup:ineq}. However, for small $\gamma$, other members of the family yield a better bound. In particular, this approach gives a quantitative improvement over the Yang-Yau inequality for all $\gamma$ except for $\gamma = 4,6,8,10,14$. We refer to Section~\ref{sec:computations} and Table~\ref{table:low_genera} for more details. 
 
 \subsection{Sketch of the proof} 
 The proof is based on a particular construction of a balanced map from the surface $\Sigma$ to the Euclidean sphere. Such maps are commonly used in the geometric optimization of eigenvalues, we refer e.g. to~\cite{BLY, Hersch, KS, LY,  NadirashviliS2} for applications to various problems. The coordinates of balanced maps are good test-functions for $\lambda_1$ provided the energy of the map is controlled. For example, in~\cite{BLY} the authors consider a particular minimal isometric embedding $A\colon \mathbb{CP}^n\to\mathbb{S}^{(n+1)^2-2}$. Precomposing $A$ with a full holomorphic map $f\colon\Sigma\to\mathbb{CP}^n$ gives the map $A_f=A\circ f$ to a sphere, whose energy is controlled by the degree of $f$. Finally, they use linear transformations on $\mathbb{CP}^n$ to arrange $A_f$ to be balanced. Among other things, this construction can be used to prove~\eqref{YY:ineq}.
 In~\cite{Ros:main article} Ros builds up on this construction by considering a perturbation $\phi_a = A_f+a H_f$, where $H_f$ is the mean curvature vector of $A_f$ and $a\in\mathbb{R}$ is a parameter. He then takes $f\colon \Sigma\to\mathbb{CP}^2$ to be the quartic realization of any non-hyperelliptic genus $3$ surface $\Sigma$ and argues that for certain range of $a$ the map $\phi_a$ can still be arranged to be balanced. Finally, it turns out that the parameter $a$ can be chosen so that the energy $E(\phi_a)<E(\phi_0)$, thus, yielding an improvement over~\eqref{YY:ineq}.
 
We generalize the argument of Ros by considering $f$ to be a full map induced by an arbitrary linear system of dimension $n$. The Brill-Noether theory provides an existence of such linear systems with degree bounded in terms of the genus of $\Sigma$. This allows to extend Ros' approach to surfaces of arbitrary genus. An additional upside is that one is free to choose $n$. In fact, $n=5$ turns out to be optimal for large $\gamma$ and corresponds to the bound~\eqref{main:ineq}.
A novel feature compared to~\cite{Ros:main article} is that our maps $f$ could have branch points. We show that the contribution of the branch points to the energy $E(\phi_a)$ has a negative sign and, thus, could be discarded in the proof of upper bound. 

In conclusion, let us provide some geometric intuition for the definition of $\phi_a$. Since the goal is to minimize the energy of a balanced map, it is natural to deform the map in the direction of the negative gradient of the energy. 
Since $A_f$ is a conformal map, this gradient coincides with the mean curvature vector. Thus, the correspondence $A_f\mapsto \phi_a$ can be seen as one step of the discrete harmonic map heat flow or, equivalently, mean curvature flow. The surprising part of the calculation in~\cite{Ros:main article} is that the energy of $\phi_a$ only depends on $a$, genus $\gamma$ and the degree of $f$. It would be interesting to see if this method can be refined further by considering two steps of the discrete flow or a continuous flow.

 \subsection*{Organization of the paper} In section~\ref{sec:holo} we prove a general upper bound for the first eigenvalue in terms of a holomorphic map to $\mathbb{CP}^n$, see Theorem~\ref{th:implicit estimate of lambda}. The content of this section is to the large extent a direct generalization of the results in~\cite{Ros:main article}. Section~\ref{sec:computations} is devoted to investigating the bound obtained in Theorem~\ref{th:implicit estimate of lambda}. Namely, we show that taking $n=5$ yields the best bound for large $\gamma$ and show that in this case the bound reduces to~\eqref{main:ineq}. Table~\ref{table:low_genera} contains the optimal values of $n$ for small $\gamma$. The section is elementary, but computationally heavy. 
 
 \subsection*{Acknowledgements} The authors would like to thank I. Polterovich and A. Penskoi for fruitful discussions. The first author is partially supported by NSF grant DMS-1363432.

\section{Eigenvalue bounds from holomorphic maps to $\CP^n$}

\label{sec:holo}

In the following $(\Sigma,g)$ denotes an orientable Riemannian surface of genus $\gamma$. The metric $g$ (together with a choice of an orientation) induces a complex structure $\Sigma$ and, thus, we view $\Sigma$ as a Riemann surface  throughout this section. The complex structure only depends on the conformal class $[g] = \{e^{2\omega}g,\,\omega\in C^\infty(M)\}$. Furthermore, our first result, Theorem~\ref{th:implicit estimate of lambda} gives a bound on 
$$
\Lambda_1(\Sigma,[g]) := \sup_{h\in [g]}\bar\lambda_1(\Sigma,g).
$$
The study of $\Lambda_1(\Sigma,[g])$ is of independent interest due to its connection with the theory of harmonic maps, see e.g.~\cite{KNPP}.

\subsection{Bracnhed holomorphic curves in complex projective spaces}
Let $f\colon \Sigma \to \CP^n$ be a holomorphic map. One can define the order of the differential at a point $p$ denoted by $\ord_p df$ as the order of the vector function $f'_z$, where $df = f'_{z}dz$ in some local coordinates on $\Sigma$ and $\CP^n.$ The total ramification of $f$ is defined as 
\begin{equation}\label{eq:total ramification}
\beta =  \sum\limits_{p} \ord_p df = \sum\limits_{p} (\ord_p f - 1). 
\end{equation}
If it happens to be nonzero, the metric $h = f^*g_{FS}$ induced by the Fubini-Study metric has conical singularities, i.e. in a local coordinate $z$ centered at $p$ the metric is given by $h = \abs{z}^{2k} \rho^{2} dz \otimes d \bar{z},$ where $\rho(0) \neq 0,\ k = \ord_p f'_{z}.$ Here we assume that $g_{FS}$ is normalized to have holomorphic sectional curvature $1$. Denote its Gauss curvature by $K$ and the volume form by $dv_h$. After some calculations, we have
\[
K dv_h = 4 \partial_{z} \partial_{\bar{z}} \log (\abs{z}^{k} \rho) dx \wedge dy = 4 \partial_{z} \partial_{\bar{z}} \log \rho \, dx \wedge dy,
\]
which means the Gauss curvature is integrable. Hence, the following formula holds, see e.g.~\cite{Gauss-Bonnet formula},
\begin{equation}\label{eq:Gauss-Bonnet formula}
\frac{1}{2 \pi} \int_{\Sigma} K dv_h=2 - 2g + \beta.
\end{equation}

        The degree of the map $f\!: \Sigma \to \CP^n$ is defined as the positive integer $d$ that corresponds to the integral homology class $f_*[\Sigma] \in H_2(\CP^n) \simeq \mathbb{Z}.$ Let $\omega_{FS}$ be the K\"ahler $2$-form associated with Fubiny-Study metric. Then $\frac{1}{4\pi}\omega_{FS}$ represents the canonical basis in integral cohomology $H^2(\CP^n).$ Since $f^*(\omega_{FS}) = dv_h,$ we conclude that
\begin{equation}\label{eq:degree-volume relation}
\area(\Sigma,h) = 4\pi d.
\end{equation}

 We are interested in finding full holomorphic maps of relatively small degree. Recall that $f\colon\Sigma\to\CP^n$ is called {\em full} if its image is not contained in a hyperplane $\CP^{n-1}\subset\CP^n$. For this purpose, we use Brill–Noether theory, see e.g. section Special Linear Systems IV in~\cite{GH}, which assures the existence of a linear system of degree $d'$ and dimension $n$ as long as $\gamma \geqslant(n+1)(\gamma+n-d').$ Removing the base points if necessary, we obtain a holomorphic map $f\colon \Sigma \to \CP^n$ of degree $d \leqslant d'.$ This implies the following lemma.
\begin{lemma}\label{lem:a holomorphic map}
For any positive $n$ there exists a full holomorphic map $f\colon \Sigma \to \CP^n$ of degree $d$ such that 
\[
	d \leqslant \ceil*{\frac{n\gamma}{n+1}} + n.
\]
\end{lemma}
 
\subsection{An embedding of $\CP^n$ into the Euclidean space}
Denote by $HM_1(m)$ the space of Hermitian operators of trace~1 acting on $\mathbb{C}^m,$ and consider the Euclidean metric  
\[
\langle A, B\rangle=2 \tr A B \quad \forall A, B \in \HM1.
\]
Then $\CP^n$ can be considered as the space of all orthogonal projectors of rank~1 in $\HM1.$ Using the matrix notation, this embedding $A\colon\CP^n\to HM_1(n+1)$  is given by the formula
\begin{equation}\label{eq:canonical embedding}
z \mapsto \frac{1}{z^* z} z z^*,
\end{equation}
where $z$ is a column of homogeneous coordinates and $z^* = \bar{z}^{t}.$ This embedding is also $U(n\!+\!1)\text{-equivariant}$, i.e. it intertwines the actions
\begin{equation}
(P, z) \mapsto P z, \qquad(P, A) \mapsto P A P^{*}
\end{equation}
on $\CP^n$ and $HM_1(n+1)$ respectively.

In what follows, we will slightly abuse the notation and denote by the same letter $A$ both the embedding $A\colon \CP^n \to \HM1$ and the points of $\CP^n$ regarded as a subset of $ \HM1$. The map $A$ has been studied in detail in~\cite{Ros: Spec geom of CR}. The following properties are of particular importance to us,
\begin{enumerate}
\item
the map $A$ is an isometry. Recall that $\CP^n$ is endowed with the Fubini-Study metric of holomorphic section curvature 1;
\item the image of $\CP^n$ is a minimal submanifold of the sphere $S^{(n+1)^2 - 2} \subset \HM1$ centered at $\frac{1}{n+1}I$ with radius $\sqrt{\frac{2n}{n+1}},$ where $I$ is the identity matrix.
\end{enumerate}

Now let $f\colon \Sigma \to \CP^n$ be a holomorphic curve. The composition of the previous embedding with $f$ produces the map $A_f\colon \Sigma \to \HM1.$ The map $A_f$ is an immersion on the complement $\mathring\Sigma$ to the set of the branched points of $f.$ Thus, on this open set $\mathring\Sigma,$ one can apply all the local formulae from differential geometry by considering $\mathring\Sigma$ as an immersed Riemannian manifold with the induced metric $h$, also denoted by $\langle \cdot, \cdot\rangle.$

Let $\sigma$ be the second fundamental form of the immersion $f\!: \mathring\Sigma \to \CP^n.$ Since $\CP^n$ with the Fubini–Study metric is a Kähler manifold, its Levi-Civita connection commutes with the complex structure, i.e. $J \nabla = \nabla J.$ As a consequence, we have 
\begin{equation}\label{eq:compat. 2d f.f. and J}
\sigma(J X, Y)=\sigma(X, J Y)=J \sigma(X, Y).
\end{equation}

An isometric immersion of a Riemannian manifold $M$ into another one $\overline{M},$ yields the relation between their sectional curvatures given as follows,
\begin{equation}\label{eq:sec.curv. relation}
K(\xi, \eta)=\overline{K}(\xi, \eta)+\frac{\langle\sigma(\xi, \xi), \sigma(\eta, \eta)\rangle-|\sigma(\xi, \eta)|^{2}}{|\xi|^{2}|\eta|^{2}-\langle\xi, \eta\rangle^{2}},
\end{equation}
where $\xi, \eta \in T_p M$ are linearly independent. Take $M = \mathring\Sigma$, $\overline{M} = \CP^n$ and 
$\eta = J\xi$ so that $\{\xi,\eta\}$ form an orthonormal basis of $T_pM$. Then one obtains that $K(\xi, \eta) = K$ is just the Gaussian curvature of $\mathring\Sigma$ and $\overline{K}(\xi, \eta)=1$ is the holomorphic sectional curvature of $\CP^n$ with the Fubini–Study metric. Applying~\eqref{eq:compat. 2d f.f. and J}, one sees that $\sigma(\xi, \xi) = -\sigma(\eta, \eta)$ and $\sigma(\xi, \eta) = J \sigma(\xi, \xi).$ This immediately implies
\begin{equation}\label{eq:Gauss equation}
K=1-\frac{1}{2}|\sigma|^{2}.
\end{equation}

Since $\HM1$ is an affine space, the coordinate-wise Hessian of $A_f$ coincides with its second fundamental form. In particular, we have 
\begin{equation}
-\Delta A_f = \tr \Hess A_f = 2H =: B,
\end{equation} 
where $H$ stands for mean curvature of the map $A_f\!: \mathring\Sigma \to \HM1.$ 
 
\begin{lemma}\label{lem:A-B relations}
The following relations hold:
\[
\left\{\begin{aligned}
&|I|^{2}=2(n+1), &\quad \langle &A_f, I\rangle=2, &\quad &|A_f|^{2}=2, \\
&\langle B, I\rangle=0, &\quad \langle &B, A_f\rangle=-2, &\quad &|B|^{2}=4, \\
&\langle\Delta_h B, A_f\rangle=-4, &\quad \langle&\Delta_h B, B\rangle=8+2|\sigma|^{2}.
\end{aligned}\right.
\]
\end{lemma}
\begin{proof}
The first row follows from the definition of the isometric embedding $\CP^n \subset \HM1.$ The next equation is obtained by noticing that $\tr B = 0$ for any $B \in  T_A\!\HM1.$ The rest follows from Lemma~3.2 in~\cite{Ros: Spec geom of Kaeh}.
\end{proof}

From formulae~\eqref{eq:Gauss-Bonnet formula},~\eqref{eq:degree-volume relation}, and~\eqref{eq:Gauss equation} one has,
\begin{equation}\label{eq:main integrals}
\left\{\begin{aligned}
&\int_{\Sigma} |\sigma|^{2} dv_h = 8\pi \left(d + \gamma - 1 - \frac{1}{2}\beta\right),\\
&\int_{\Sigma} 1 dv_h = 4\pi d.
\end{aligned}\right.
\end{equation}

Fix a point $a \in \mathbb{R}$ and consider the map $\phi_a\!: \mathring\Sigma \to \HM1,$
\[
\phi_a(A) = A + a B(A).
\]
\begin{lemma}\label{lem:phi_a properties}
The image of $\phi_a$ lies in the sphere with center at $\frac{1}{n+1} I$, namely,
\[
\left|\phi_{a}-\frac{1}{n+1} I\right|^{2} = (2a - 1)^2 + \frac{n-1}{n+1}.
\]
The energy of $\phi_a$ is
\[
\int_{\Sigma}\left|d\phi_{a}\right|_h^{2} d v_h = 8\pi d \left[(2a - 1)^{2}+2 a^{2}\delta \right],
\]
where $\delta = 1 + \frac{g-1-\frac{1}{2}\beta}{d}\geqslant 0.$ 

\par In particular, the energy remains invariant under projective transformations of $\Sigma$ in $\CP^n.$ Furthermore, $\phi_a$ regarded as a vector-valued function belongs to the Sobolev space $W^{1,2}(\Sigma, g)$ for all smooth metrics $g$  compatible with the complex structure on $\Sigma$.
\end{lemma}
\begin{proof}
Both conclusion are obtained using Lemma~\ref{lem:A-B relations}. For the first one we have
\begin{equation}
\begin{split}
&\left|\phi_{a}-\frac{1}{n+1} I\right|^{2}=\left|a B+\left(A_f-\frac{1}{n+1} I\right)\right|^{2}= \\
&=a^{2}|B|^{2}+2 a\langle B, A_f\rangle + 
\left|A_f-\frac{1}{n+1} I\right|^{2} = 4a^{2} - 4a + 1 + \frac{n-1}{n+1}.
\end{split}
\end{equation}

For the second, we observe that
\begin{equation}\label{eq:nabla_phi}
\begin{split}
\int_\Sigma\left|d\phi_{a}\right|_h^{2}\,dv_h&=\int_\Sigma\left\langle\Delta_h \phi_{a}, \phi_{a}\right\rangle\,dv_h=\int_\Sigma\langle -B+a \Delta_h B, A_f+a B\rangle\,dv_h=\\
&=\int_\Sigma 2(1-2a)^2 + 2a^2 \abs{\sigma}^2\,dv_h.
\end{split}
\end{equation}
An application of formulae~\eqref{eq:main integrals} completes the proof of the identities.

To show that $\phi_a\in W^{1,2}(\Sigma,g)$ it is sufficient to observe that the Dirichlet integral is conformally invariant and that $\phi_a$ is a bounded function.
\end{proof}

\subsection{The center of mass} Let us denote by $\mathcal{H}$ the convex hull of $\CP^n \subset \HM1.$ Recalling the fact that $\CP^n$ is embedded to $\HM1$ as a set of all one-dimensional orthogonal projectors and the fact that Hermitian operators are diagonalizable, one has:
\begin{itemize}
\item
$\mathcal{H}=\left\{A \in \HM1 | \ A \geqslant 0\right\},$ where $A \geqslant 0$ means that $A$ is positive semi-definite;
\item
$int \, \mathcal{H}=\left\{A \in \HM1 | \ A > 0\right\};$
\item
$\partial \mathcal{H}=\{A \in \mathcal{H} \ | \rk A \leqslant n\}.$ 
\end{itemize}
Using these facts, we can prove the following lemma.
\begin{lemma}\label{lem:center mass}
The distance between $\partial \mathcal{H}$ and the point $\frac{1}{n+1}I$ equals $\sqrt{\frac{2}{n(n+1)}}.$
\end{lemma}
\begin{proof}
Let $A \in \partial \mathcal{H}$ be a point that realizes the distance. As $U(n\!+\!1)$ acts by isometries, we may suppose that $A$ is diagonal, i.e.
\begin{equation}
A = \begin{pmatrix}
0 		& 	0     	& \cdots & 0 		\\
0 		& 	a_1 	& \cdots & 0 		\\
\vdots  & 	\vdots  & \ddots & \vdots \\
0 		& 	0 		& \cdots & a_n
\end{pmatrix}, \quad a_i \geqslant 0, \ \sum_{i} a_i = 1.
\end{equation}
Remark that $\frac{1}{n+1}I$ is the center of the $n$-dimensional simplex, $A$ lies on a one of its faces, and so $A$ must be the center of the face, i.e. $a_i = \frac{1}{n}.$ Thus, 
\begin{equation}
\abs{A - \frac{1}{n+1}I}^2 = 2\left(\frac{1}{(n+1)^2} + n \frac{1}{(n+1)^2 n^2} \right) = \frac{2}{n(n+1)}.
\end{equation}
\end{proof}

Any point $P \in int \, \mathcal{H}\subset\HM1$ defines a projective transformation of $\CP^n\!.$ If the map $A_f\!: \Sigma \to \HM1$ corresponds to the holomorphic map $f\!: \Sigma \to \CP^n,$ then we denote by $A_{f_P}$ the map corresponding to $f_P=P \circ f.$ Let $B_P$ be its mean curvature vector.

Given a metric $g$ on $\Sigma$ we define the map $\Phi_{a}\colon int \, \mathcal{H} \to \HM1$ by the formula 
\begin{equation}
\Phi_{a}(P)=\frac{1}{\operatorname{Area}\left(\Sigma, g\right)} \int_{\Sigma}\left(A_{f_P}+a B_{P}\right) d v_g.
\end{equation}
To each transformation $P$ the map $\Phi_a$ assigns the center of mass of $A_{f_P}+a B_{P}$ with respect to the metric $g$. Our goal is to show that there exists $P_0$ such that $\Phi_a(P_0) = \frac{1}{n+1}I$.

We consider $a$ as a parameter. Lemma~\ref{lem:phi_a properties} implies that $A_{f_P}+a B_{P}$ is bounded provided $a$ bounded. It follows from the dominated convergence theorem that the map $\Phi$ depends continuously on $a$ and $P.$ The case $a=0$ is studied in~\cite{BLY} and it turns out that the map $\Phi_{0}$ possesses the following properties,
\begin{itemize}
\item
it extends to the continuous map $\Phi_{0}\!: \mathcal{H} \to \HM1;$
\item
$\Phi_{0}(\partial\mathcal{H}) \subset \partial\mathcal{H},$ and the restriction of this extension $\Phi_{0}{\big|}_{\partial\mathcal{H}}$ has non-zero degree.
\end{itemize}
\begin{remark}
This is the only point of the proof that requires $f$ to be full. The fullness assumption is used in~\cite{BLY} to prove the properties above. 
\end{remark}
\begin{lemma}\label{lem:a useful point}
If $\abs{a} < \frac{1}{\sqrt{2n(n+1)}},$ then there exists $P_0\in int \, \mathcal{H}$ such that $\Phi_a(P_0) = \frac{1}{n+1}I.$
\end{lemma}
\begin{proof}
We "shrink" $\mathcal{H}$ a little bit and call it $\mathcal{H}_{\varepsilon},$
\[
\mathcal{H}_{\varepsilon}=\left\{(1-\varepsilon) P+\varepsilon\frac{1}{n+1}I \, \big| \, P \in \mathcal{H}\right\} \subset int \, \mathcal{H}.
\]
We claim that $\Phi_a(\partial \mathcal{H}_{\varepsilon})$ does not contain $\frac{1}{n+1} I$ when $\varepsilon$ is small enough. Indeed, if $\Phi_a(P_\varepsilon) = \frac{1}{n+1}I$, then one has 
\[
0 = \Phi_a(P_\varepsilon) - \tfrac{1}{n+1}I = \left(\Phi_0(P_\varepsilon) - \frac{1}{n+1}I\right) + \frac{a}{\operatorname{Area}\left(\Sigma, g\right)} \int_{\Sigma}B_{P_\varepsilon} d\mu_g,
\]
so by Lemma~\ref{lem:A-B relations} 
\[
\abs{\Phi_0(P_\varepsilon) - \frac{1}{n+1}I} \leqslant 2\abs{a} < dist \! \brr{\partial \mathcal{H}, \frac{1}{n+1}I},
\]
but $\Phi_0(P_\varepsilon) \to \partial \mathcal{H}$ as $\varepsilon \to 0,$ which will be a contradiction.

It remains to notice that $\Phi_{a}\!: \partial\mathcal{H}_\varepsilon \to \HM1 \backslash \brc{\tfrac{1}{n+1}I}$ homotopic to $\Phi_{0}\colon\partial\mathcal{H} \to \partial\mathcal{H},$ therefore, has the same degree, which is non-zero by~\cite{BLY}. Hence, $\Phi_a(\mathcal{H}_\varepsilon)$ must contain $\frac{1}{n+1}I,$ and this concludes the proof.
\end{proof}

\subsection{Upper bound on $\lambda_1$.}
Finally, we are ready to prove the upper bound on $\Lambda_1(\Sigma,[g])$.
Recall that 
\begin{equation}
\label{eq:Rayleigh quotient}
\lambda_{1}(\Sigma, g)=\inf \left\{\frac{\displaystyle\int_{\Sigma}|d\varphi|^{2} d v_g}{\displaystyle\int_{\Sigma} \varphi^{2} dv_g} \,,\, \varphi \in W^{1,2}(\Sigma, g)\backslash \{0\}, \ \int_{\Sigma} \varphi dv_g = 0\right\}.
\end{equation}

The following theorem is a direct generalization of the bound of Ros in~\cite{Ros:main article}.
\begin{theorem}
\label{th:implicit estimate of lambda}
Let $(\Sigma,g)$ be a compact oriented Riemannian surface of genus $\gamma$ endowed with the compatible complex structure. Suppose that there exists a full holomorphic map $f\colon\Sigma \to \CP^n$  with total ramification $\beta = \sum\limits_{p} (\ord_p f - 1).$ 
Then for any $\abs{a}\leqslant \frac{1}{\sqrt{2n(n+1)}}$ one has 
\begin{equation}
\label{general:ineq}
\Lambda_1(\Sigma,[g])\leqslant 8\pi \deg(f) \left(1 + \frac{2a^{2}\delta - \frac{n-1}{n+1}}{(2a - 1)^2 + \frac{n-1}{n+1}}\right),
\end{equation}
where $\delta = 1 + \frac{\gamma-1-\frac{1}{2}\beta}{\deg(f)}$.
\end{theorem}

\begin{proof}
Fix $\abs{a} < \frac{1}{\sqrt{2n(n+1)}}$ and consider a metric $g \in [g].$ We use Lemma~\ref{lem:a useful point} and replace $f$ with $f_P.$ It has the same degree and total ramification, but in addition, one has
\[
\int_{\Sigma}\left(\phi_{a}-\frac{1}{n+1} I\right) dv_g=0.
\]
This allows us to use the coordinates of $\left(\phi_{a}-\tfrac{1}{n+1}\right)$ as test-functions in~\eqref{eq:Rayleigh quotient}, which are in $W^{1,2}(\Sigma, g)$ by Lemma~\ref{lem:phi_a properties}. Let $G(a)$ be the r.h.s in~\eqref{general:ineq}.
Thus, by identities in  Lemma~\ref{lem:phi_a properties} one obtains
\[
\lambda_{1}(\Sigma, g) \leqslant \frac{\displaystyle\int_{\Sigma}\left|d\phi_{a}\right|_g^{2} dv_g}{\displaystyle\int_{\Sigma}\left|\phi_{a}-\frac{1}{n+1} I\right|^{2} dv_g} = \frac{\displaystyle\int_{\Sigma}\left|d\phi_{a}\right|_h^{2} d v_h}{\displaystyle\int_{\Sigma}\left|\phi_{a}-\frac{1}{n+1} I\right|^{2} d v_g}=\frac{G\left(a\right)}{\area(\Sigma, g)},
\]
where we used taht $\left|\phi_{a}-\tfrac{1}{n+1} I\right|$ is constant and the energy of $\phi_{a}$ does not change within the class of conformal metrics. The case of $\abs{a} = \frac{1}{\sqrt{2n(n+1)}}$ follows from the continuity of $F(a)$.
\end{proof}

\section{The proof of the main theorem}

\label{sec:computations}

To effectively apply inequality~\eqref{general:ineq} one needs to find full holomorphic maps $f\colon \Sigma\to\mathbb{CP}^n$ of low degree. Such maps are given by Lemma~\ref{lem:a holomorphic map}. Note that the r.h.s of~\eqref{general:ineq} is increasing in $\deg(f)$ and decreasing in $\beta$, therefore one has the following proposition.
\begin{proposition}
\label{ngamma:prop}
Let $\Sigma_\gamma$ be an orientable surface of genus $\gamma$. Then for any $n\in \mathbb{Z}$, $n>0$  and any $|a|\leqslant \frac{1}{\sqrt{2n(n+1)}}$ one has 
$$
\Lambda_1(\Sigma_\gamma) \leqslant 8\pi d\left(1 + \frac{2a^{2}\delta - \frac{n-1}{n+1}}{(2a - 1)^2 + \frac{n-1}{n+1}}\right) =: F(a,n,\gamma),
$$
where $d= d(n,\gamma) = \ceil*{\frac{n\gamma}{n+1}} + n$, $\delta = \delta(n,\gamma) = 1 + \frac{\gamma-1}{d}$. 
\end{proposition}

The remainder of this section is devoted to choosing the values $(a,n)$ that lead to the best bound for each fixed $\gamma$. 
One first observes that for $n=1$ the minimum of $F(a,1,\gamma)$ is achieved for $a=0$ and the resulting bound is the Yang-Yau bound~\eqref{YY:ineq}. In the following we assume $n\geqslant 2$.

\begin{lemma}
For fixed $(n,\gamma)$, the function $F(a,n,\gamma)$ has exactly two critical points on $\mathbb{R}$, a local minimum at 
$$
a_{\min} = a_{\min}(n,\gamma) = \frac{\xi-\sqrt{\xi^2-2(n^2-1)\delta}}{2 \delta(n+1)}
$$
and a local maximum at
$$
a_{\max} = a_{\max}(n,\gamma) = \frac{\xi+\sqrt{\xi^2-2(n^2-1)\delta}}{2 \delta(n+1)},
$$
where $\xi = n\delta + (n-1)$. Furthermore, $a_{\max}>\frac{1}{\sqrt{2n(n+1)}}$. Thus, the minimum of $F(a,n,\gamma)$ on $\left[-\frac{1}{\sqrt{2n(n+1)}}, \frac{1}{\sqrt{2n(n+1)}}\right]$ is achieved either at $a=a_{\min}$ or at $a=\frac{1}{\sqrt{2n(n+1)}}$.
\end{lemma}
\begin{proof}
The derivative $\partial_aF(a,n,\gamma)$ has the form $\frac{p_{n,\gamma(a)}}{q_{n,\gamma}(a)^2}$, where $p_{n,\gamma}(a)$ is a quadratic polynomial in $a$ and $q_{n,\gamma}(a) = (2a - 1)^2 + \frac{n-1}{n+1}$ never vanishes. Therefore, the first assertion follows from the formula for the roots of a quadratic polynomial.

Recall that we are assuming $n\geqslant 2$, therefore, one has
$$
a_{\max}\geqslant \frac{\xi}{2\delta(n+1)}>\frac{n}{2(n+1)}\geqslant \frac{1}{3}>\frac{1}{2\sqrt{3}}>\frac{1}{\sqrt{2n(n+1)}}.
$$
\end{proof}

\begin{lemma} For any $\gamma\geqslant 0$ and $n\geqslant 5$ one has
$a_{\min}(n,\gamma) > \frac{1}{\sqrt{2n(n+1)}}$. Furthermore, $a_{\min}<\frac{1}{\sqrt{2n(n+1)}}$ provided that 
(a) $n=2$, $\gamma\geqslant 0$; (b) $n=3$, $\gamma\geqslant 4$; (c) $n=4$, $\gamma\geqslant 40$.
\end{lemma}
\begin{proof}
The condition $a_{\min} < \frac{1}{\sqrt{2n(n+1)}}$ can be rewritten as follows
\[
\xi - \sqrt{2\frac{n+1}{n}}\delta < \sqrt{\xi^2-2(n^2-1)\delta}.
\]
Provided $n \geqslant 2,$ one can take squares, and direct calculation yield
\begin{equation}
\label{ineq: delta_1 constrain}
\frac{\sqrt{n} (n-1)\brr{\sqrt{n(n+1)} - \sqrt{2}}}{n\sqrt{2n} - \sqrt{n+1}} < \delta.
\end{equation}

Note that for any $\gamma\geqslant 0$ the inequality $\delta(n,\gamma)\leqslant 2+\frac{1}{n}$ holds. Thus, for $n\geqslant 5$ one has
\begin{equation}
\begin{split}
&\frac{\sqrt{n} (n-1)\brr{\sqrt{n(n+1)} - \sqrt{2}}}{n\sqrt{2n} - \sqrt{n+1}} \geqslant \frac{(n-1)\brr{\sqrt{n(n+1)} - \sqrt{2}}}{\sqrt{2}n-1}\geqslant\\
&\geqslant \frac{\sqrt{n(n+1)} - \sqrt{2}}{\sqrt{2}}\geqslant 2+\frac{1}{5}\geqslant 2 + \frac{1}{n}\geqslant\delta(n,\gamma)
\end{split}
\end{equation}
for any $\gamma\geqslant 0$. As a result, by~\eqref{ineq: delta_1 constrain} one obtains $a_{\min}>\frac{1}{\sqrt{2n(n+1)}}$ for $n\geqslant 5$.

Similarly, for all $\gamma\geqslant 0$ one has $\delta(n,\gamma)\geqslant 1+\frac{(n+1)(\gamma-1)}{n\gamma + (n+1)(n+2)}$. For a fixed $n$ the r.h.s is an increasing function of $\gamma$, therefore, as long as
\begin{equation}
\label{eq:1}
\frac{\sqrt{n} (n-1)\brr{\sqrt{n(n+1)} - \sqrt{2}}}{n\sqrt{2n} - \sqrt{n+1}} < 1+\frac{(n+1)(\gamma_0-1)}{n\gamma_0 + (n+1)(n+2)}
\end{equation}
the condition $a_{\min}(n,\gamma)<\frac{1}{\sqrt{2n(n+1)}}$ is satisfied for $\gamma\geqslant \gamma_0$. Finally, a direct calculation shows that~\eqref{eq:1} holds for $(n,\gamma_0) = (2,0),(3,4),(4,40)$.
\end{proof}
\begin{remark}
By a direct computation using~\eqref{ineq: delta_1 constrain} one can see that, in fact, $a_{\min}<\frac{1}{\sqrt{2n(n+1)}}$ for $n=3,$ $\gamma\geqslant 3$ and $n=4,$ $\gamma\geqslant 30$. This observation is not necessary for the following arguments, since the low genus case is treated separately in Table~\ref{table:low_genera}.
\end{remark}

Having determined the optimal value of $a$, we define
\begin{equation}
F_n(\gamma) = 
\begin{cases}
F\left(\frac{1}{\sqrt{2n(n+1)}},n,\gamma\right),&\text{if $n\geqslant 5$;}\\
F(a_{\min}(n,\gamma),n,\gamma), &\text{if} \ n=2; n=3,\gamma\geqslant 4; n=4,\gamma\geqslant 40;\\
F(0,1,\gamma) = 8\pi\left(\ceil*{\frac{\gamma}{2}}+1\right),&\text{if $n=1$,}
\end{cases}
\end{equation} 
where $F_1(\gamma)$ is the r.h.s in the Yang-Yau inequality~\eqref{YY:ineq}.

A direct computation yields that 
\begin{equation}
\label{eq:F5}
F_5(\gamma) = \frac{2\pi}{13-\sqrt{15}} \brr{\gamma + \brr{33-4\sqrt{15}}\ceil*{\frac{5\gamma}{6}} + 4\brr{41-5\sqrt{15}}}.
\end{equation}
Estimating $\ceil*{x}\leqslant x+1$ in~\eqref{eq:F5} gives the following useful estimate,
\begin{equation}
\label{ineq:F5}
\frac{F_5(\gamma)}{8\pi}\leqslant  \left(\frac{89}{6(4\sqrt{15}-52)}+\frac{5}{6}\right)\gamma+\left(\frac{115}{4 \sqrt{15}-52}+6\right)  =: a\gamma + b,
\end{equation}
where $a\leqslant 0.4271$, $b\leqslant 2.8501$.

The following lemma states that taking $n\geqslant 6$ is never optimal.
\begin{lemma}
For any $\gamma\geqslant 0$ and any $n> 5$ one has
$$
F_n(\gamma)> F_5(\gamma).
$$
\end{lemma}
\begin{proof}
Since 
$$
d(n,\gamma) = \ceil*{\frac{n\gamma}{n+1}} + n\geqslant \frac{n\gamma}{n+1} + n
$$
one has that for $n>5$
$$
\frac{F_n(\gamma)}{8\pi} \geqslant \frac{n\gamma}{n+1}+n+\frac{1}{2}  \frac{\gamma \left(2 n - n^{2}+\frac{1}{n+1}\right)-(n-1)^{2} (n+1)}{n^{2}-\sqrt{2n(n+1)}+1} =:a_n\gamma + b_n.
$$
 To complete the proof we estimate $a_n$, $b_n$ separately to show that $a_n>a$ and $b_n>b$ for $n\geqslant 6$, where $a,b$ are as in~\eqref{ineq:F5}.

One has
\begin{equation}
\label{eq:an}
a_n = \frac{n}{n+1} + \frac{2 n - n^{2}+\frac{1}{n+1}}{2(n^{2}-\sqrt{2n(n+1)}+1)}.
\end{equation}
Thus,
\begin{equation}
\label{ineq:an}
a_n\geqslant\frac{n}{n+1} - \frac{n^2-2n}{2(n^2+1-2(n+1))} = \frac{1}{2} -\frac{1}{n+1} + \frac{1}{n^2-2n-1}\geqslant \frac{1}{2} - \frac{1}{n+1}.
\end{equation}
The r.h.s of~\eqref{ineq:an} is increasing in $n$. As a result, it is straightforward to compute that~\eqref{ineq:an} implies $a_n>a$ for $n\geqslant 13$. For $5<n<13$ one proves $a_n>a$ by a direct computation using~\eqref{eq:an}.

Similarly, one has 
\begin{equation}
\label{eq:bn}
b_n = n - \frac{(n-1)^2(n+1)}{2(n^{2}-\sqrt{2n(n+1)}+1)}.
\end{equation}
Thus,
\begin{equation}
\label{ineq:bn}
b_n\geqslant n - \frac{(n-1)^2(n+1)}{2(n^{2}-2n-1)} = \frac{n-3}{2} - \frac{6n+2}{2(n^2-2n-1)} \geqslant \frac{n-5}{2},
\end{equation}
where in the last inequality we used $n\geqslant 6$. As a result,~\eqref{ineq:bn} implies $b_n> b$ for $n\geqslant 11$. For $5<n<11$ one proves $b_n>b$ by a direct computation using~\eqref{eq:bn}. 

\end{proof}

\begin{table}
	{\renewcommand{\arraystretch}{1.4}
	\begin{tabular}{|c | l |}
	\hline
	Value of $n$ & Genera $\gamma$, for which $n$ is optimal \\
	\hline
	1 & 4, 6, 8, 10, 14 \\
	\hline
	2 & 3, 7, 9, 18, 19 \\
	\hline
	3 & 5, 12, 13, 16, 17, 20, 24, 28, 29, 32 \\
	\hline
	4 & 11, 15, 21, 22, 23, 25-27, 30, 31, 33 -- 41, 45 - 47, 50 -- 53,   \\
	   & 55 - 61, 65, 70, 71, 75 - 77, 80 - 82, 95, 100, 101\\
	\hline
	5 & 42 -- 44, 48, 49, 54, 62-69, 72 -- 74, 78, 79, 83 -- 94, 96 -- 99,  \\
	   & $\gamma\geqslant 102$\\
	\hline
	\end{tabular}
	}
	\caption{ Optimal values of $n$ for low genera. The value $n=1$ corresponds to the Yang-Yau unequality, i.e. for $\gamma=4,6,8,10,14$ our results do not improve on~\eqref{YY:ineq}.
	}
\label{table:low_genera}
\end{table}

Finally, we investigate the range $1\leqslant n \leqslant 5$.

\begin{lemma}
\label{lemma:n<5}
One has the following
\begin{enumerate}
\item If $\gamma\geqslant 25$, then $F_1(\gamma)>F_5(\gamma)$;
\item If $\gamma\geqslant 41$, then $F_2(\gamma) >F_5(\gamma)$;
\item If $\gamma\geqslant 109$, then $F_3(\gamma) >F_5(\gamma)$;
\item If $\gamma\geqslant 389$, then $F_4(\gamma) >F_5(\gamma)$.
\end{enumerate}
\end{lemma}
\begin{proof}
Let us first deal with $\gamma = 1$. One obtains $F_1(\gamma)\geqslant 4\pi(\gamma+2)$ and a direct comparison with~\eqref{ineq:F5} gives $F_1(\gamma)>F_5(\gamma)$ for $\gamma\geqslant 26$. The case $\gamma = 25$ can be checked directly.

If $x_0$ is a critical point of a function $\frac{f(x)}{g(x)}$, then $\frac{f(x_0)}{g(x_0)} = \frac{f'(x_0)}{g'(x_0)}$. Therefore, for $n=2,3,4$ one has 
\begin{equation}\label{eq:F(a_min)}
F_n(\gamma)=8 \pi d\left(1-\frac{\delta a_{\min }}{1-2 a_{\min }}\right).
\end{equation}
Recall that 
$$
\frac{n\gamma}{n+1} + (n+1)\leqslant d(n,\gamma) = \ceil*{\frac{n\gamma}{n+1}} + n+1\leqslant  \frac{n\gamma}{n+1} + (n+2),
$$
$$
\delta(n,\gamma) = 1 + \frac{\gamma-1}{d(n,\gamma)};\quad \xi(n,\gamma) = n\delta(n,\gamma) + n-1.
$$

{\bf Case 1: $n=2$}. One has
$$
\delta_-(\gamma) = \frac{5\gamma+9}{2\gamma+12}\leqslant\delta\leqslant \frac{5\gamma+6}{2\gamma+9}\leqslant \frac{5}{2}= \delta_+;\qquad \xi\leqslant 2\delta_+ + 1 =  6 = \xi_+.
$$
Then one has
\begin{equation}
a_{\min} \leqslant \frac{\xi_+-\sqrt{\xi_+^2-6\delta_+}}{6\delta_-(\gamma)} = \frac{(6-\sqrt{21})(\gamma+6)}{3(5\gamma+9)} = a_+(\gamma).
\end{equation}
We observe that $a_+(\gamma)$ is a decreasing function of $\gamma$. Therefore, by~\eqref{eq:F(a_min)} one has
$$
\frac{F_2(\gamma)}{8\pi}\geqslant \left(\frac{2\gamma}{3}+3\right)\left(1 -\frac{\delta_+a_+(\gamma)}{1-2a_+(\gamma)} \right)\geqslant \left(\frac{2\gamma}{3}+3\right)\left(1 -\frac{\delta_+a_+(\gamma_0)}{1-2a_+(\gamma_0)} \right)
$$
for all $\gamma\geqslant \gamma_0$. Taking $\gamma_0=41$, a direct computation combined with the expression~\eqref{ineq:F5} yields the claim.

{\bf Case 2: $n=3$}. One has
$$
\delta_-(\gamma) = \frac{7\gamma + 16}{3\gamma+20}\leqslant \delta\leqslant \frac{7\gamma+12}{3\gamma+16}\leqslant \frac{7}{3} = \delta_+;\qquad \xi\leqslant 3\delta_+ + 2 = 9 = \xi_+.
$$
Then one has
$$
a_{\min}\leqslant \frac{\xi_+-\sqrt{\xi_+^2-16\delta_+}}{8\delta_-(\gamma)}=\left(9 - \sqrt{81 - \frac{112}{3}}\right)\frac{3\gamma+20}{8(7\gamma+16)} = a_+(\gamma).
$$
We observe that $a_+(\gamma)$ is a decreasing function of $\gamma$. Therefore, by~\eqref{eq:F(a_min)} one has
$$
\frac{F_3(\gamma)}{8\pi}\geqslant \left(\frac{3\gamma}{4}+4\right)\left(1 -\frac{\delta_+a_+(\gamma_0)}{1-2a_+(\gamma_0)} \right)
$$ 
for all $\gamma\geqslant\gamma_0$. Taking $\gamma_0=109$, a direct computation combined with the expression~\eqref{ineq:F5} yields the claim.

{\bf Case 3: $n=4$}. One has
$$
\delta_-(\gamma) = \frac{9\gamma + 25}{4\gamma+30}\leqslant \delta\leqslant \frac{9\gamma+20}{4\gamma+25}\leqslant\frac{9}{4} = \delta_+;\quad \xi\leqslant 4\delta_+ + 3 = 12.
$$
Then one has
$$
a_{\min}\leqslant \frac{\xi_+-\sqrt{\xi_+^2-30\delta_+}}{10\delta_-(\gamma)} = \left(12 - \sqrt{144 - \frac{135}{2}}\right)\frac{4\gamma+30}{10(9\gamma + 25)} = a_+(\gamma).
$$
We observe that $a_+(\gamma)$ is a decreasing function of $\gamma$. Therefore, by~\eqref{eq:F(a_min)} one has
$$
\frac{F_4(\gamma)}{8\pi}\geqslant \left(\frac{4\gamma}{5}+5\right)\left(1 -\frac{\delta_+a_+(\gamma_0)}{1-2a_+(\gamma_0)} \right)
$$ 
for all $\gamma\geqslant \gamma_0$. Taking $\gamma_0=389$, a direct computation combined with the expression~\eqref{ineq:F5} yields the claim.
\end{proof}

The results of this section implies that choosing $n=5$ is optimal for large $\gamma$. The optimal values of $n$ for small $\gamma$ are collected in Table~\ref{table:low_genera}.

%
%

\end{document}